\documentclass[11pt]{amsart}
\usepackage{standard}
\usepackage{enumitem}
\usepackage[all]{xy}

\newcommand{\E}{\mathcal{E}}
\newcommand{\lag}{\mathcal{L}}
\newcommand{\lie}{\mathcal{L}}
\newcommand{\tlag}{\widetilde{\mathcal{L}}}
\newcommand{\tT}{\widetilde{T}}
\newcommand{\talpha}{\widetilde{\alpha}}

\newcommand{\tlambda}{\widetilde{\lambda}}
\newcommand{\hamvf}{\mathsf{H}}
\newcommand{\ombar}{\overline{\omega}}
\newcommand{\omtilde}{\widetilde{\omega}}
\newcommand{\WFh}{\WF_h}
\newcommand{\Psih}{\Psi_h}
\newcommand{\torus}{\mathbb{T}}
\newcommand{\cU}{\mathcal{U}}

\newcommand{\be}{\mathbf{e}}
\newcommand{\bff}{\mathbf{f}}
\newcommand{\ocal}{\mathcal{O}}

\DeclareMathOperator{\grph}{Graph}
\DeclareMathOperator{\Ran}{Ran}
\DeclareMathOperator{\Id}{Id}
\DeclareMathOperator{\Opw}{Op_W}

\title[Non-concentration of quasimodes]{Non-concentration of
  quasimodes for integrable systems}
\date{\today}
\author{Jared Wunsch}

\thanks{The author is grateful to Michael Koehn and Steve Zelditch for
  raising the question answered here.  He thanks Zelditch for many
  instructive conversations on Lagrangian quasimodes and semiclassical
  analysis.  He is also grateful to Nalini Anantharaman for helpful
  discussions and to Andr\'as Vasy for comments on the manuscript, as
  well as to an anonymous referee.  This work was partially supported
  by NSF grants DMS-0700318 and DMS-1001463.}

\begin{document}
\begin{abstract}
  We consider the possible concentration in phase space of a sequence
  of eigenfunctions (or, more generally, a quasimode) of an operator
  whose principal symbol has completely integrable Hamilton flow.  The
  semiclassical wavefront set $\WF_h$ of such a sequence is invariant
  under the Hamilton flow. In principle this may allow concentration
  of $\WF_h$ along a single closed orbit if all frequencies of the
  flow are rationally related.  We show that, subject to
  non-degeneracy hypotheses, this concentration may not in fact occur.
  Indeed, in the two-dimensional case, we show that $\WF_h$ must fill
  out an entire Lagrangian torus.  The main tools are the spreading
  of Lagrangian regularity previously shown by the author,
  and an analysis of higher order transport equations satisfied by the
  principal symbol of a Lagrangian quasimode.  These yield a unique
  continuation theorem for the principal symbol of Lagrangian
  quasimode, which is the principal new result of the paper.
\end{abstract}

\maketitle

\section{Introduction}
A central question in spectral geometry is how a sequence of
eigenfunctions of the Laplace-Beltrami operator on an $n$-dimensional
manifold $X$ may concentrate in phase space.  The \emph{semiclassical
  wavefront set} or \emph{frequency set}, here denoted $\WF_h,$ is a
closed set measuring the locations in phase space (i.e., $T^*X$) where
a sequence of functions is non-negligible; for a sequence of
eigenfunctions of the Laplacian, $\WF_h$ is known to be invariant
under the geodesic flow.  In the case when the geodesic flow is
completely integrable, this leaves open the possibility that $\WF_h$
may concentrate on a single closed orbit or some other small set
invariant under the geodesic flow.  The results of this paper put
limitations on this possible concentration.  We show that in many
circumstances, concentration may not occur on the smallest possible
set allowed by standard propagation arguments, which is to say, a
single closed orbit of the bicharacteristic flow.  These results can
be viewed as analogues of results of Bourgain, Jakobson, Maci\`a and
Anantharaman-Maci\`a in the case when $X$ is a flat torus (and of more
recent work of Burq-Zworski on Schr\"odinger operators on $2$-tori),
extended to the broader context of (nondegenerate) completely
integrable systems.

The results of this paper apply somewhat more generally than to
sequences of eigenfunctions: they are equally applicable to
approximate eigenfunctions or \emph{quasimodes} (see
\cite{Arnold:Modes}, \cite{Colin:Quasi}).  We write such approximate
solutions in the formalism of semiclassical analysis: for instance,
instead of having a sequence of approximate eigenfunctions of the (non-negative)
Laplace-Beltrami operator:
$$
(\Lap-\lambda_k^2) u_k=O(\lambda_k^{-\infty}),
$$
we may set $h=1/\lambda_k,$ suppress the index, and write
$$
(h^2\Lap-1) u =O(h^\infty).
$$
This is the notation we employ below.

Our first results show how the results obtained by the author
\cite{Wunsch:Integrable} (and refined by Vasy and the author \cite{twomicro,twomicroerratum})
on the spreading of Lagrangian regularity of
quasimodes for quantizations of classically integrable systems can be
easily applied to obtain results on the nonconcentration of semiclassical
wavefront set in the same setting.
Consider a semiclassical pseudodifferential operator $P$ with
semiclassical principal symbol $p,$ acting on
half-densities.\footnote{For a discussion of semiclassical
  pseudodifferential operators, we refer the reader to
  \cite{Dimassi-Sjostrand} or \cite{Evans-Zworski}.}
  We assume that 
\begin{enumerate}[label=(\Alph*)]
\item\label{hyp1} $p=\sigma_h(P)$ is real.
\item
\label{hyp2}
The subprincipal symbol of  $P$ (which is well defined for an
operator on half-densities) is a real constant on\footnote{It will in
  fact suffice for the subprincipal symbol to be real everywhere on
  $T^*X$ and constant on the Lagrangian $\lag$ discussed below.  In
  what follows we will
  indicate in footnotes where the proof varies in the case of this
  weaker hypothesis.} $T^*X.$
\item\label{hyp3}
The
bicharacteristic flow of the Hamilton vector field $\hamvf_p$ is
\emph{completely integrable}.
\end{enumerate}
The integrability hypothesis means that there exist action-angle
variables $(I_1,\dots,I_n,\theta_1,\dots,\theta_n)$, i.e.\ symplectic
coordinates $I \in \RR^n,$ $\theta \in
\RR^n/\ZZ^n$ such that $p=p(I)$ is independent of $\theta.$ In fact \emph{we
only need to assume nondegeneracy of these coordinates in some region
of interest in phase space} (in a neighborhood of the Lagrangian torus
introduced below) rather than globally, where it is unlikely to hold
in general.  The assumption on subprincipal symbols is certainly
satisfied for geometric Schr\"odinger operators $h^2\Lap_g+V(x)$ with
$g$ a Riemannian metric and $V$ a real potential, as the subprincipaal
symbol vanishes identically; for
more details on this hypothesis, we refer the reader to \S5 of \cite{twomicro}.

The Arnol'd-Liouville tori of the system are the sets where the $I$
variables are held constant.  We consider one such torus $\lag.$
 Let $$\omega_i = \partial
p/\partial I_i$$ and $$\omega_{ij} = \partial^2 p/\partial I_i \partial I_j.$$
Let $\ombar_i$ and $\ombar_{ij}$ denote the corresponding quantities
restricted to $\lag$ (where they are constant).  We make the
assumption, standard in KAM theory, that
\begin{enumerate}[resume,label=(\Alph*)]
\item\label{hyp4} 
The system is (in a
neighborhood of $\lag$) \emph{isoenergetically nondegenerate;}
\end{enumerate}
this means that the matrix
\begin{equation}\label{bigmatrix}
\Omega=
\begin{pmatrix}
\ombar_{11}& \dots & \ombar_{1n} & \ombar_1 \\
\vdots & \ddots & \vdots & \vdots\\
\ombar_{n1} & \dots & \ombar_{nn} & \ombar_n\\
\ombar_{1} & \dots & \ombar_{n} & 0\\
\end{pmatrix}
\end{equation}
is nonsingular, or equivalently, that the map
from the energy surface to the
projectivization of the frequencies
$$
\{p=0\} \ni I \mapsto [\omega_1(I):\dots: \omega_n(I)] \in \mathbb{RP}^n
$$
is a local diffeomorphism.

Now we consider a normalized quasimode of $P,$ concentrated on $\lag,$
i.e.\ a family of distributions $u=u(x;h)$ satisfying
\begin{enumerate}[resume,label=(\Alph*)]
\item \label{hyp5}
$\displaystyle P u =O(h^\infty),\quad \norm{u}_{L^2}=1,\quad \WFh u
\subset \lag$.
\end{enumerate}
For example $u=u(x;h)$ may be a family of exact, normalized
Schr\"odinger eigenfunctions in the nullspace of $P=h^2\Lap+V-E(h),$
with $E(h) \sim E_0 + h E_1 +\dots$ as $h\downarrow 0;$ more
generally, it may be a family of approximate eigenfunctions.  We
recall that the \emph{semiclassical wavefront set} or \emph{frequency
  set} of $u$ is defined as the closed set
$$
\WFh u =\big\{\rho \in T^*X: \exists A \in \Psih(X),\ \sigma_h(A)(\rho)\neq 0,\
A u =O(h^\infty)\big\}^\complement.
$$
This set is well known, by the semiclassical analogue of the
Duistermaat-H\"ormander theorem on propagation of singularities, to be
invariant under $\hamvf_p.$ Thus if there exist no rational linear
relations among the $\ombar_i$, then as a closed invariant set, $\WFh
u$ must fill out the whole of the torus $\lag.$ (It cannot be empty,
as that would contradict $L^2$-normalization.)  The results of
Bourgain and Jakobson \cite{Jakobson:tori} in the case of the
Laplacian on flat tori show that even on completely rational tori,
however, the semiclassical limit measure of a sequence of
eigenfunctions must project to the base to be absolutely continuous;
Maci\`a \cite{Macia1} Anantharaman-Maci\`a \cite{Anantharaman-Macia1}
show analogous results for more general limit measures arising from
Schr\"odinger flow.  We prove similar results in the setting of
$\WFh,$ in the more general setting described above.  Estimates
roughly of the form that we employ here were previously used by
Burq-Zworski \cite{Burq-Zworski:Control} in showing that a sequence of
eigenfunctions on the Bunimovich stadium cannot concentrate along a
single bouncing-ball orbit, and generalized to the general setting of
integrable systems by the author \cite{Wunsch:Integrable} and further
by the author and Vasy \cite{twomicro, twomicroerratum}.  Very recent
results of Burq-Zworski \cite{Burq-Zworski} also yield stronger,
control-theoretic, estimates in the special case of a Schr\"odinger
operator $\Lap+V$ on a two-dimensional torus.

To begin, we recall a key result from
\cite{Wunsch:Integrable}.\footnote{The hypothesis used in
  \cite{Wunsch:Integrable} was that of vanishing subprincipal symbol;
    the use of weaker hypotheses is discussed in
    \cite{twomicro,twomicroerratum}.}  This result describes how
  Lagrangian regularity with respect to $\lag$ is forced to spread on
  $\lag.$  (The definition of Lagrangian regularity is discussed below.)
\begin{theorem}[\cite{Wunsch:Integrable}]\label{theorem:0}
Let $P$ and $u$ satisfy hypotheses \ref{hyp1}--\ref{hyp5}. 
\begin{enumerate}
\item If the dimension $n=2$ and $u$ enjoys Lagrangian regularity at
  some point on $\lag,$ then $u$ is Lagrangian with respect to $\lag.$
\item
 If the
$\ombar_i$ are all rationally related, and $u$ enjoys Lagrangian
regularity outside a single closed bicharacteristic $\gamma\subset
\lag,$ then $u$ is Lagrangian with respect to $\lag.$  
\end{enumerate}
\end{theorem}
Semiclassical Lagrangian distributions are a special
class of distributions with $\WF_h$ lying only on $\lag.$ There are
two principal ways to describe this class.  One characterization is
simply that $u$ has (locally near each point) an oscillatory integral
representation of a well-understood form, which we discuss below.  The
second characterization is that $u$ has iterated regularity
with respect to pseudodifferential operators characteristic on $\lag.$
We describe these definitions below in \S\ref{section:lagrangian}; for
further details, we refer the reader to the paper of Alexandrova
\cite{Alexandrova:Semiclassical}, where the results of the
H\"ormander-Melrose theory \cite{Hormander:vol4} are adapted to the
semiclassical setting.  The simplest and most instructive example of a
semiclassical Lagrangian distribution is just a family of the form
\begin{equation}\label{lag:ex}
h^{-s} a(x;h) e^{i\phi(x)/h}
\end{equation}
where $s \in \RR,$ $a$ is in $\CI,$ uniformly in $h$ and $\phi \in \CI.$  This is a
Lagrangian distribution with respect to
$$
\lag=\grph(d\phi) \subset T^*X.
$$
Indeed, if the projection of $\lag$ to $X$ is a
diffeomorphism, every Lagrangian distribution with respect to $\lag$
is locally of this form.  When the projection is not a diffeomorphism,
we need instead to employ oscillatory integral representations, i.e.,
integral superpositions of expressions of the form \ref{lag:ex}---see
\eqref{localform} below.

As a simple consequence of Theorem~\ref{theorem:0}, we are able to prove:
\begin{theorem}\label{theorem:1}
Let $P$ and $u$ satisfy hypotheses \ref{hyp1}--\ref{hyp5}.
\begin{enumerate} 
\item
 Let $n=2.$
 Then $\WFh
u$ has nonempty interior (in the relative topology of $\lag$).
\item 
 If the
$\ombar_i$ are all rationally related, then $\WF_h u$ cannot be a
subset of a single closed bicharacteristic.
\end{enumerate}
\end{theorem}

The upshot of these results is that a normalized quasimode may not
concentrate on too small a subset of a Lagrangian torus.  We now state
an additional result, with stronger hypotheses, that implies that
indeed the quasimode must be supported on the \emph{whole}
invariant torus.

The necessary stronger hypothesis is:
\begin{enumerate}[resume,label=(\Alph*)]
\item \label{hyp6}
The system is (in a neighborhood of $\lag$) \emph{quasi-convex} in the sense of Nekhoroshev.
\end{enumerate}
This means that the Hessian $\pa^2 p/\pa I_i \pa I_j$ is
strictly positive definite on the fixed energy surface $p=0$ (in a
neighborhood of $\lag$)---see \cite{Marco-Sauzin} for details of the
use of this hypothesis in proving exponential stability of orbits for
perturbed integrable systems.
\emph{Note that the hypothesis of quasi-convexity implies isoenergetic nondegeneracy (hypothesis
\ref{hyp4}).}  The main new result in this paper is then the following.
\begin{theorem}\label{theorem:2}
Let $P$ and $u$ satisfy \ref{hyp1}--\ref{hyp3} and \ref{hyp5}--\ref{hyp6}.
Assume that $u$ is
Lagrangian with respect to $\lag.$
Then the support of the principal symbol of $u$ is all of
$\lag.$  The same conclusion still holds if \ref{hyp5} is weakened to
merely $Pu=O(h^{2+\delta})$ for some $\delta>0.$
\end{theorem}
\begin{remark}
The principal symbol of a Lagrangian distribution is locally a
  half-density on $\lag$ which characterizes its leading order
  behavior; in the example \eqref{lag:ex}, it can be taken to be the
  amplitude $a$, modulo $O(h)$.  The global description of the
  principal symbol is more involved, as one must take into account
  both Maslov factors and the cohomology class of the canonical
  one-form restricted to $\lag.$ For accounts of this construction, we
  refer the reader to work of Duistermaat
  \cite{Duistermaat:Oscillatory} and Bates and Weinstein
  \cite{Bates-Weinstein}.

  A priori, we have not assumed that the Lagrangian distribution is
  \emph{classical}, i.e.\ enjoys a power series expansion in $h:$ $a
  \sim a_0+ha_1 +\dots.$ Thus, the support theorem as stated merely
  tells us that there is no open set $\ocal\subset \lag$ on which the
  localization of $a$ is $O(h).$ In the special case of classical
  Lagrangians, though, this implies that the support of $a_0$ is
  all of $\lag.$

In the general, non-classical case, we in fact show more than
nonvanishing modulo $O(h)$: to wit, we obtain the stronger statement that
there do not exist an open set $\ocal \subset\lag$ and a sequence
$h_j\downarrow 0$ with $\sigma_h(a)(x;h_j)\to 0$ pointwise a.e.\ on
$\ocal.$
\end{remark}

By Theorem~\ref{theorem:2}, we immediately obtain a very strong
non-concentration result in two dimensions, subject to the quasi-convexity hypothesis:
\begin{corollary} 
Let $P$ and $u$ satisfy \ref{hyp1}--\ref{hyp3} and
\ref{hyp5}--\ref{hyp6}.  Let $n=2.$
Then $\WFh u=\lag.$
\end{corollary}
\begin{remark}
More generally, we may relax the 
assumption $\WFh u \subset \lag:$ if we merely have $\WFh u \cap \lag
\neq \emptyset,$ then it follows that $\lag \subset \WFh u.$
\end{remark}

\section{Proof of quasimode nonconcentration}\label{section:lagrangian}
We begin by recalling in more detail the notion of Lagrangian
regularity.  Throughout this paper, we will let $\Psi_h(X)$ denote the
algebra of semiclassical pseudodifferential operators on the manifold $X$ (acting on
half-densities) obtained locally by quantization of Kohn-Nirenberg
symbols, as described for instance in \S9 of \cite{Evans-Zworski}.

We say that $u$ is
Lagrangian with respect to $h^s L^2$ at $\rho \in \lag$ if there
exists a neighborhood $U$ of $\rho$ in $T^*X$ such that for all $k \in \NN$ and
all $A_1,\dots,A_k \in \Psi_h(X)$ with $\sigma_h(A_j)=0$ on $\lag,$
and $\WF'(A_J) \subset U,$ we
have
$$
h^{-k} A_1\cdots A_k u \in h^s L^2.
$$
In other words, $u$ enjoys ``iterated regularity''
under the application of operators of the form $h^{-1} A$ with $A$
characteristic on $\lag$ and microsupported near $\rho.$

We now restate Theorem~\ref{theorem:0} slightly more precisely: In
\cite{Wunsch:Integrable} (see also \cite{twomicro,twomicroerratum}) it
was shown that, subject to the above hypotheses, if $u$ is in $L^2$
and is Lagrangian\footnote{It suffices to assume Lagrangian regularity
  at this point with respect to $h^{-t}L^2$ for \emph{any} $t\in \RR,$
  as by interpolation with $u \in L^2$ we automatically obtain
  Lagrangian regularity with respect to $h^{-\delta} L^2$ for all
  $\delta>0.$} in an annular region (i.e., a hollow tube) surrounding
a single closed bicharacteristic on a rational Lagrangian torus, then
for every $\ep>0$ it is Lagrangian with respect to $h^{-\ep}L^2$ on
$\gamma$ as well.  Likewise, in the special case $n=2,$ the same
argument proved that if $u$ is in $L^2$ and Lagrangian \emph{at a
  single point} on $\lag,$ then $u$ is Lagrangian with respect to
$h^{-\ep}L^2$ \emph{globally on $\lag.$}

We now prove the first part of Theorem~\ref{theorem:1}.
Either all of $\lag$ lies in $\WFh u$ or there exists $\rho \in \lag$
such that $\rho \notin \WFh u.$
In the latter case, then \emph{a fortiori}
$u$ is Lagrangian at $\rho.$  Thus, by the results of \cite{twomicro},
$u$ is Lagrangian on all of $\lag$ with respect to $h^{-\ep}L^2$ for
all $\ep>0.$  Now by the semiclassical analog of the H\"ormander-Melrose theory of Lagrangian
distributions (see Alexandrova \cite{Alexandrova:Semiclassical}) this
means that microlocally near $\lag,$ $u$ can be written as an oscillatory
integral
\begin{equation}\label{localform}
u  = h^{-\ep+N/2} \int_{\RR^N} a(x,\xi;h) e^{i \phi(x,\xi)/h} \, d\xi
\end{equation}
for some $N \in \NN$ determined by the geometry of $\lag$ (in
particular by the local form of its projection to the base) and
with an amplitude $a \in \CcI,$ uniformly in $h.$  Here $\phi$
parametrizes $\lag$ in the sense that
\begin{equation}\label{parametrization}
\lag= \{(x,d_x\phi): (x,\xi) \in C\}\equiv \Phi(C).
\end{equation}
where
\begin{equation}\label{C}
C =\{(x,\xi): d_\xi \phi=0\}.
\end{equation}
By stationary phase, for $x,\xi \in C,$
$$
\Phi(x,\xi) \in \WFh u \Longleftrightarrow a(\bullet;h) \neq
O(h^\infty) \text{ in a neighborhood of } (x,\xi).
$$
By smoothness of $a,$ the set of such points is the closure of an open
set.  (Recall that the
semiclassical wavefront set cannot be empty, as that would contradict
the $L^2$ normalization.)  This concludes the proof of the first part
of the theorem.

The proof of the second part is analogous.  By the hypotheses, we have
Lagrangian regularity everywhere but along $\gamma.$
Theorem~\ref{theorem:0} then allows us to conclude global Lagrangian
regularity, which, by the argument above, shows that $\WFh u$ has
nonempty interior, and in particular, cannot have been a subset of
$\gamma$ after all.\qed

\section{Lagrangian quasimodes}

In this section, we prove Theorem~\ref{theorem:2}.

We recall that the principal symbol $\sigma_h(u)$ of a semiclassical Lagrangian
distribution $u \in h^{-s} L^2$ given locally by
\begin{equation}
u  = h^{-s+N/2} \int_{\RR^N} a(x,\xi;h) e^{i \phi(x,\xi)/h} \, d\xi
\end{equation}
can be \emph{locally} identified with the half-density $a,$ restricted
to the manifold $C$ given by \eqref{C} (which is, in turn,
diffeomorphically identified with $\lag$ via $\Phi$ defined in
\eqref{parametrization}), modulo $O(h).$ By standard results in the
calculus of Lagrangians, we know that $\sigma_h(u)$ is necessarily
invariant under $\lie_{\hamvf_p}+i p'$ with $p'$ denoting (the
operation of multiplication by) a subprincipal symbol of $P.$ (Indeed,
this same invariance property holds globally, where we interpret the
symbol as a section of the tensor product of $\Omega^{1/2}$ with a
flat complex line bundle.)

\subsection{The model case}
We begin our finer analysis of $\sigma_h(u)$ by considering the model
case $X=\torus_x^n =\RR_x^n/\ZZ^n,$ with the action-angle variables given by
$$
I_j=\xi_j,\quad \theta_j=x_j
$$
and the Lagrangian torus given by the
zero section:
$$
\lag_0=\{\xi=0\} \subset T^*(\torus_x^n).
$$
We will use $\abs{dx_1\cdots dx_n}^{1/2}$ to trivialize the
half-density bundle; note that this has vanishing Lie derivative along
constant coefficient vector fields in $x,$ hence we will identify
symbols with functions and the action of $\lie_{\hamvf_p}$ with that
of $\hamvf_p.$

Since the principal symbol of $P$ is assumed to be a function only of $\xi,$ vanishing on $\lag_0,$ its Taylor
expansion in the $\xi$ variables reads
$$
p = \sum \ombar_j\xi_j + \ombar_{ij}\xi_i\xi_j+ O(\abs{\xi}^3)
$$
and hence, since the subprincipal symbol of $P$ is a assumed to be a
real constant\footnote{If the subprincipal symbol is only constant on
  $\lag,$ there is an extra term in this expression of the form
  $h^2\sum \kappa_i(x) D_i.$}  $c,$
$$
P = h\sum \ombar_j D_j+hc + h^2 \sum \ombar_{ij}D_iD_j + \sum h^3 D_i
D_j D_k Q_{ijk}+h^2R
$$
where $Q_{ijk}, R\in \Psi_h(X).$ The hypothesis of quasi-convexity of $p$ near $\lag_0$
means that the matrix $\ombar_{ij}$ is positive definite on the
orthocomplement of the span of the vector $\sum \ombar_j \pa_j.$  Note
that the term $\sum
\ombar_j D_j$ is $i^{-1} \hamvf_p$ in this setting, and we
will use this latter notation as well.

We may further Taylor expand the principal symbol of $R$ into
$$
r(x) +\sum \xi_i r_i(x,\xi),
$$
hence we may express
$$
h^2 R = h^2 r(x) + h^3 \sum R_i D_i+ h^3 E
$$
with $E, R_i \in \Psi_h(X)$ and $r(x)$ denoting the multiplication operator
by the function of the same name.  Plugging this into our expression for $P$ we now
obtain
\begin{multline}
P = h\sum \ombar_j D_j+hc + h^2 \sum \ombar_{ij}D_iD_j+h^2 r(x)\\ +
h^3 \sum D_i D_j D_kQ_{ijk}+h^3\sum D_i R_i+ h^3\widetilde{R}
\end{multline}
with $\widetilde{R} \in \Psi_h(X).$  We will collectively write
the last three terms as $O(h^3)$ below.

We remark\footnote{Indeed, this remark will justify our $O(h^3)$ notation
  just introduced.} that in this model situation, the iterated regularity
definition of Lagrangian distributions and the definition by
oscillatory integrals actually coincide: a parametrization of $\lag_0$
is by the phase function $\phi=0,$ and $u$ is Lagrangian with respect
to $h^{-\ep} L^2$ if and only if
$$
u=u(x;h) \in \CI(\torus^n),\text{ with } \norm{\pa_x^\alpha u} \lesssim
h^{-\ep} \text{ for all } \alpha,
$$
i.e.\ $u$ is $h^{-\ep}$ times a function that is smooth in $X,$ uniformly in
$h.$  In this
situation, the distribution $u$ and its total symbol coincide.  The
principal symbol of $u$ as a Lagrangian distribution with respect to
$h^{-\ep} L^2$ is simply the equivalence class.
$$
\sigma_h(u) =  u \bmod O(h^{1-\ep}).
$$

We now prove a unique continuation theorem for
the principal symbol of a Lagrangian quasimode of $P$ in the model
setting.  This will constitute the main step in the proof of
Theorem~\ref{theorem:2}, with the remainder of the proof being a
conjugation to this model problem.

\begin{proposition}\label{proposition:normformcase}
Let $\lag_0$ be the zero section of $T^*(\torus^n)$ and
  $P$ as described above.
  If $Pu=O(h^{2+\delta})$ with $\delta>0$ and $u$ is Lagrangian with respect to $\lag_0$
  and $L^2$ normalized, then there do not exist $\ocal \subset \lag_0$
  open and $h_j \downarrow 0$ with $\sigma_h(u)(x;h_j) \to 0$
  pointwise a.e.\ on  $\ocal.$
\end{proposition}
\begin{proof}
As discussed above, we have
\begin{equation}\label{modelequation}
\begin{aligned}
P &= h\sum \ombar_j D_j+hc + h^2 \sum \ombar_{ij}D_iD_j+h^2 r(x)
+O(h^3)\\
&= h\big(-i\hamvf_p+c) + h^2 \sum \ombar_{ij}D_iD_j+h^2 r(x) +O(h^3)
\end{aligned}
\end{equation}
Certainly then the principal symbol, $\sigma_h(u)=u \bmod O(h^{1-\ep})$, is
annihilated by
$$
-i \hamvf_p +c,
$$
hence if it is nonvanishing at a point $\rho \in \lag_0,$ it is also
nonvanishing along the whole orbit $\gamma_\rho$ of $\rho$ under the
flow along $\hamvf_p.$ We let $T$ denote the closure of the orbit
through one such $\rho$; $T$ is necessarily a sub-torus of dimension $k
\geq 1.$  Identifying $\torus^n$ with $\RR^n/\ZZ^n$ we may lift $T$ to
$\RR^n$ and translate coordinates to obtain a vector subspace $V
\subset \RR^n$ of dimension $k.$  Letting $L=V \cap \ZZ^n$ we obtain
a sub-lattice of $\ZZ^n$, with $\ZZ^n/L$ torsion-free.  Hence by
taking a basis of the quotient module we may
extend a basis $\be_1,\dots,\be_k$ of $L$ to a basis
$\be_1,\dots,\be_k,\bff_1,\dots,\bff_{n-k}$ of $\ZZ^n.$  Now identifying
$$
\torus^n = \RR^n/(\ZZ \be_1+ \dots + \ZZ \be_k+ \ZZ
\bff_1+ \dots+ \ZZ\bff_{n-k}),
$$
we may split
$$
\torus^n \cong T \times T'
$$
with
$$
T = (\RR \be_1+ \dots+ \RR \be_k)/(\ZZ \be_1+ \dots+ \ZZ \be_k),
$$
$$
T'=(\RR \bff_1+ \dots+ \RR \bff_{n-k})/(\ZZ \bff_1+ \dots+ \ZZ \bff_{n-k}).
$$
We may thus introduce new coordinates
$x=(y_1,\dots,y_k,z_1,\dots,z_{n-k})$ on $\torus^n$ with $T$ defined
by $z=0,$ $y \in \RR^k/\ZZ^k$ and $T'$ defined by $z=0$ with $y \in
\RR^{n-k}/\ZZ^{n-k}.$
On the torus
$T$ (and its translates $\{z=z_0\},$ for $z_0 \in T'$), the Hamilton flow has dense orbits,
hence the space of solutions $v$ of functions on $T$ satisfying
\begin{equation}\label{cohom}
(-i \hamvf_p +c) v=0
\end{equation}
has dimension at most $1,$ since specifying $v$ on a single point then determines
$v$ on a dense set in $T.$  Equivalently, Fourier analyzing in
$y\in T$, we may write
$$
v(y) = \sum_{\alpha \in \ZZ^k} \hat v(\alpha) e_\alpha(y)
$$
with $e_\alpha(y) =e^{2 \pi i \alpha \cdot y}.$
Since $\hamvf_p=\sum \ombar_i \pa_{x_i}$ is tangent to $T,$ and since its orbit
closure is dense in $T,$ we may write it in our new coordinates as
$$
\sum \ombar_i \pa_{x_i}= \sum \omtilde_j \pa_{y_j}
$$
for some $\omtilde_i,$ $i=1,\dots, k,$ \emph{with no rational relation
  holding among the $\omtilde_i.$}  Then
$$
(-i \hamvf_p +c) v(y)=0 \Longleftrightarrow \sum_{\alpha \in \ZZ^k} (\omtilde \cdot \alpha+c)\hat v(\alpha) e_\alpha(y)=0.
$$
By the irrationality assumption, there can be at most one value of $\alpha$ such
that $\omtilde \cdot \alpha+c=0,$ hence $v$ can only have at most this
one Fourier mode.  

If the dimension of the solution space of \eqref{cohom} is $0,$ then since the
principal symbol must satisfy \eqref{cohom}, there cannot be an
$L^2$-normalized quasimode.  Thus, we assume that the dimension of
smooth solutions to \eqref{cohom} is $1,$ hence that there is exactly one
frequency vector $\alpha_0$ satisfying $\omtilde \cdot \alpha_0+c=0.$

Now we decompose both sides of \eqref{modelequation} in Fourier modes
along $T,$ i.e.\ in the $y$ variables.  We let
$$
u = \sum e_\alpha(y) \hat u_\alpha(z;h)
$$
denote the Fourier series of $u.$  We also note that the second
order constant coefficient operator
$$
Q=\sum \ombar_{ij} D_{x_i} D_{x_j}
$$
becomes, under the change to $y,z$ variables, a new operator of the
same form, which is \emph{elliptic in the $z$ variables} as a
consequence of the quasi-convexity hypothesis (which tells us that
$\sum \ombar_{ij} D_i D_j$ is elliptic on the orthocomplement of
$\hamvf_p$).  We may rewrite this operator as
$$
Q=\sum \rho^1_{ij} D_{y_i} D_{y_j} + \sum \rho^2_{ij} D_{y_i} D_{z_j}+ \sum_{i,j=1}^{n-k} \Omega_{ij} D_{z_i} D_{z_j}
$$
with the matrix $\Omega$ positive definite.
Thus, it acts on
$$
u=\sum e_\alpha(y) \hat u_\alpha(z;h),
$$
as
\begin{equation}\label{ellipticpde}
\begin{aligned}
Qu &= \sum_{\alpha \in \ZZ^k} e_\alpha(y) \bigg[ \sum_{i,j=1}^{n-k} \Omega_{ij} D_{z_i}
D_{z_j} +\sum \gamma_i(\alpha) D_{z_i}+\rho(\alpha)\bigg] (\hat u_\alpha)
\\ &\equiv\sum_\alpha e_\alpha(y) Q_\alpha (\hat u_\alpha)
\end{aligned}
\end{equation}
with $\gamma_i,\rho$ depending linearly resp.\ quadratically on $\alpha.$
Each operator $Q_\alpha$ is an elliptic constant coefficient operator
in $z.$  We also decompose
$$
r(x) = r(y,z) = \sum e_\alpha(y) \hat r_\alpha(z).
$$

Thus, our Fourier analysis of \eqref{modelequation} (for an
$O(h^{2+\delta})$ quasimode) finally yields
\begin{equation}\label{foobar}
 \sum e_\alpha(y) \big( h(\omtilde \cdot \alpha +c) \hat u_\alpha 
+ h^2Q_\alpha \hat u_\alpha \big) + h^2 \sum_{\alpha,\beta}
e_{\alpha+\beta}(y) \hat r_\beta \hat u_\alpha+ O(h^{3-\ep})=O(h^{2+\delta}).
\end{equation}
(where we have used the fact that overall, $u= O(h^{-\ep})$).
Examining this equation modulo $O(h^{2-\ep})$ reveals that all terms
$\hat u_\alpha$ with $\alpha \neq \alpha_0$ are $O(h^{1-\ep}).$  Thus,
we focus attention on the coefficient of  $e_{\alpha_0}.$  The term
$\omtilde \cdot \alpha +c$ in
\eqref{foobar} then vanishes, and the terms in the discrete
convolution are $O(h^{3-\ep})$ except when $\beta=0,$
$\alpha=\alpha_0;$ thus, we obtain, by examining the coefficient of
$e_{\alpha_0}$ in \eqref{foobar} modulo $O(h^{2+\delta}),$
$$
h^2 (Q_{\alpha_0} + \hat r_0)\hat u_{\alpha_0}= O(h^{2+\delta}).
$$
Hence for some $g(z;h),$ uniformly bounded as $h \downarrow 0,$
$$
(Q_{\alpha_0} + \hat r_0)\hat u_{\alpha_0}(z;h)= h^{\delta} g(z;h).
$$
Let $L$ denote the elliptic operator\footnote{In the case of the
  weaker assumption on the subprincipal symbol, this operator would
  have first order terms in it as well; the second order part would be
  unchanged, however.} $$L=Q_{\alpha_0}+\hat r_0$$ on
$T' \cong \RR^{n-k}/\ZZ^{n-k}.$  Since it is elliptic on a compact manifold, $L$ has finite-dimensional nullspace
on $L^2(T'),$ with a partial inverse $G$ satisfying
$$
LG =\pi_{\Ran(L)}.
$$
hence for each $h,$
$$
L(h^\delta G(g)) = h^\delta g,
$$
and we conclude that
$$
\hat u_{\alpha_0}(z;h)= \big(h^{\delta} G(g)(z;h) + v(z;h)\big),
$$
with 
$$
L v(z;h) = 0 \text{ on } T'
$$
for each $h.$  Note that there exists $h_0$ such that for $h<h_0,$
$\norm{v}_{L^2(T')}>1/2,$ since $u$ was $L^2$-normalized, and Fourier
modes other than the $\alpha_0$ mode have decaying mass.

Now suppose that there exists $\ocal' \subset T'$ with $\hat
u_{\alpha_0}(z;h) \to 0$ pointwise a.e.\ for $z \in \ocal'$ as $h=h_j\downarrow 0.$
Then we must have
\begin{equation}\label{convtozero}
v(z;h_j) \to 0
\end{equation}
for all $z \in \ocal'.$ As the nullspace of $L$ is finite-dimensional,
and its elements enjoy the property of unique continuation (see, e.g.,
Theorem~17.2.6 of \cite{Hormander:vol3}), there exists $c>0$ such
that \begin{equation}\label{uniquecontinuation} Lf(z)=0,\
  \norm{f}_{L^2(T')}\geq 1/2 \Longrightarrow\int_{\ocal'}\abs{f}^2\,
  dz \geq c.\end{equation} By Dominated Convergence (which we may
apply since elements of the nullspace of $L$ are uniformly bounded above), this is a
contradiction with \eqref{convtozero}.  Thus, such a set $\ocal'$
cannot exist.

Now since all Fourier coefficients with $\alpha \neq \alpha_0$ vanish as $O(h^{1-\ep})$, we may
take $u_{\alpha_0}(z;h) e_{\alpha_0}(y)$ to be a representative of
$\sigma_h(u);$ by the above considerations, we may take
$$
\sigma_h(u) = v(z;h) e_{\alpha_0}(y)+O(h^\delta).
$$
Thus, there cannot exist an open set  in $\lag_0$ on which
$\sigma_h(u)\to 0$ pointwise a.e.\ along
any sequence $h_j \downarrow 0,$ as this would entail the existence of
$\ocal'\subset T'$ on which $v \to 0.$
\end{proof}

\subsection{The general case (proof of Theorem~\ref{theorem:2})}

Finally, we turn to the general case of the theorem, in which $X$ is
arbitrary and $\lag$ is any Liouville torus in $T^*X.$ The proof is by
conjugating to the normal form studied in
Proposition~\ref{proposition:normformcase}, with the difficulty that there
do exist obstructions to the global existence of semiclassical Fourier
integral operators quantizing a given symplectomorphism (microlocally,
there is no obstruction).  In
particular, the Bohr-Sommerfeld-Maslov quantization conditions are known to
obstruct this process.  Fortunately, in the situation at hand,
we know a priori that there exists a nontrivial Lagrangian
distribution supported along one of our orbit closures, and
this ensures that the cohomological obstruction vanishes on the
homology classes represented in the orbit closure.

As $u$ is normalized, there is some point at which it has nonvanishing
principal symbol, hence by invariance of the symbol under
$-i\lie_{\hamvf_p} + c,$ there is some orbit closure $T$ along which
the principal symbol of $u$ as a Lagrangian distribution with respect
to $h^{-\ep} L^2$ is nonvanishing.  This principal symbol takes values
in $L \otimes \Omega^{1/2}\otimes \E$ where $L$ is the Maslov bundle
over $\lag,$ $\Omega^{1/2}$ is the half-density bundle, and $\E$ is
the \emph{pre-quantum} line bundle (see \cite{Bates-Weinstein} \S4.1,
4.4, where this object is denoted $\iota^*\E_{M,\hbar}$).  As observed
by Maslov \cite{Maslov} and further amplified, for instance, in
\cite{Duistermaat:Oscillatory}, \cite{Bates-Weinstein}, the existence
of a nonvanishing section of this bundle has a topological
implication: if it exists globally on $\lag,$ we obtain a
periodicity:\footnote{We note that there exists a distinction, at this
  point, between the case when $h \downarrow 0$ is taken to be a
  continuous parameter versus a discrete family; in the former case,
  \eqref{maslov} then entails that $\lambda$ and $(1/4)\alpha$ must
  separately vanish, while in the latter there is the possibility that
  for a given sequence of $h\downarrow 0,$ \eqref{maslov} is satisfied
  with a nontrivial left-hand side.}
\begin{equation}\label{maslov}
\frac{\lambda}{2\pi h} \equiv \frac 14 \alpha \bmod H^1(\lag; \ZZ).
\end{equation}
Here $\lambda$ denotes the cohomology class of the canonical one-form
$\xi \, dx$ restricted to $\lag$ (the ``Liouville class''), and $\alpha$ denotes the Maslov class in
$H^1(\lag; \ZZ)$  (cf.\ equation (1.5.3) of
\cite{Duistermaat:Oscillatory}).  By contrast, in the case at hand,
the section only exists locally along $T$ (and so along its translates
as well), hence
\eqref{maslov} holds but only \emph{when paired with cycles in $H_1(T)
  \subset H_1(\lag),$} i.e., if $\iota$ denotes the inclusion $T \to \lag,$
\begin{equation}\label{maslov2}
\frac{\iota^*\lambda}{2\pi h} \equiv \frac 14 \iota^* \alpha \bmod H^1(T; \ZZ).
\end{equation}

By our (local) integrability hypothesis there exists a (locally
defined) symplectomorphism
$\kappa$ from $T^*X$ to $T^*(\torus^n)$ mapping $\lag$ to $\lag_0,$ the zero section of
$T^*(\torus^n):$ we simply set
$$
x\circ \kappa=\theta,\quad \xi\circ \kappa=I.
$$
In general, we may or may not be able to
quantize such a symplectomorphism to a unitary FIO.  In the case at
hand, however, it turns out that we may do so microlocally along
$T.$  To prove this, we proceed as follows.  Let $$\Lambda =
\grph(\kappa)' \subset T^*(X \times \torus^n).$$ (This is a Lagrangian
manifold; the prime denotes inversion of dual variables in the second
factor.)  
Let $\tlag$ denote the restriction of the twisted graph to $\lag:$
$$
\tlag = \{(\rho, \kappa(\rho)'): \rho \in \lag\} \subset \Lambda;
$$
let $\tT$ denote the further restriction to the torus $T:$
$$
\tT = \{(\rho, \kappa(\rho)'): \rho \in T\} \subset \tlag.
$$
and let $\cU$ be a tubular neighborhood of $\tT$ in $\Lambda.$

Now let $\tlambda$ and $\talpha$ denote the
Liouville class and Maslov class of $\Lambda.$  By Example~5.26 of
\cite{Bates-Weinstein}, if $s_\lag$ denotes the embedding $\lag \hookrightarrow
\tlag\subset \Lambda,$ then $s_\lag$ induces an isomorphism between
$H^*(\lag)$ and $H^*(\Lambda),$ and this isomorphism preserves Maslov classes:
$$
\alpha = s_\lag^*(\talpha).
$$
Also, if $\iota_{\tlag}$ denotes the inclusion of $\tlag$ into
$\Lambda$ and $\pi$ the projection from $T^*X \times T^*(\torus^n)$ to
the left factor, since the right projection of $\tlag$ has range in
the zero section, we have
$$
\iota_{\tlag}^*(\tlambda) = \iota_{\tlag}^* \circ \pi^* (\xi\cdot dx) =
\pi_\lag^*(\lambda)= (\hat{s}_{\lag}^*)^{-1}\lambda
$$
where $\pi_\lag$ is the projection from $\tlag \subset \Lambda$ to
$\lag$ and $\hat{s}_\lag$ is the diffeomorphism from $\lag$ to $\tlag$
(hence $\pi_\lag \circ \hat{s}_\lag=\Id$).  As a result, we obtain
$$
\lambda=s_\lag^*(\tlambda).
$$
\begin{figure}[t]\caption{The maps among $\lag, \tlag, \Lambda.$}
\begin{displaymath}
\xymatrix{
 & \tlag\ar[d]^{\iota_{\tlag}} \ar[dl]^{\!\!\pi_{\lag}} \\ \lag  \ar@/^/[ur]^{\hat{s}_{\lag}}  \ar[r]_{s_{\lag}} & \Lambda}
\end{displaymath}
\end{figure}
Thus, since $\pi_\lag$ and $s_\lag$ induce isomorphisms on cohomology,
\eqref{maslov2} shows that for any cycle $\gamma$ on $\tlag$ lying in
$\tT,$
\begin{equation}\label{maslov3}
\big(\frac{\tlambda}{2\pi h} - \frac 14
\talpha,\gamma)=\big(\frac{s_\lag^*(\tlambda)}{2\pi h} - \frac 14
s_\lag^*(\talpha),(\pi_\lag)_*(\gamma)\big) = \big(\frac{\lambda}{2\pi h} - \frac 14
\alpha,(\pi_\lag)_*(\gamma)\big) \in \ZZ.
\end{equation}
Since any cycle $\gamma
\in H_1(\Lambda;\ZZ)$ lying in $\cU$ is
homologous to a cycle in $\tT\subset \tlag$
we in fact obtain
\begin{equation}\label{maslov4}
\frac{\iota_{\cU}^*\tlambda}{2\pi h} \equiv \frac 14 \iota_{\cU}^* \talpha \bmod H^1(\cU; \ZZ),
\end{equation}
with $\iota_{\cU}$ denoting the inclusion of $\cU$ into $\Lambda.$

As a consequence of the quantization condition \eqref{maslov4}, there
exists a Maslov canonical operator microlocally defined over a
neighborhood of $\tT$ i.e., we can find a Lagrangian distribution
$U_0$ on $X\times \torus^n,$ microsupported in a neighborhood of $\tT$
and Lagrangian with respect to $\Lambda,$ with principal symbol of norm
$1$ on a sub-neighborhood of $\tT.$ Viewing $U_0$ as the Schwartz
kernel of an operator $\mathcal{D'}(X) \to \mathcal{D'}(\torus^n),$ 
there exists $R \in \Psi_h(X)$ such that
$$
U_0^*U_0  = \Id +h R
$$
when \emph{acting on the space of distributions microsupported near $T.$}
Consequently, on such distributions, the h-FIO
$$
U = U_0( U_0^*U_0 )^{-1/2}
$$
acts unitarily (i.e., is a partial isometry in a microlocal sense).

In fact, it turns out that we can refine the above argument to demand
a little more of $U$ than mere microlocal unitarity.  In \cite[Theorem
2.4]{Hitrik-Sjostrand1}, Hitrik-Sj\"ostrand show that we may construct
$U$ so as to obey an improved Egorov theorem: if $a$ is a
semiclassical symbol on $T^*X$ and $\Opw$ denotes the Weyl
quantization, then we may achieve
\begin{equation}\label{betteregorov}
U \Opw(a) U^* = \Opw(a\circ \kappa^{-1}+O(h^2)),
\end{equation}
i.e.\ the Egorov theorem holds to one order better than is usual.  In
general, such a $U$ cannot be taken to be single valued---it is,
rather, Floquet periodic; however as described above, the
obstruction to its global single-valued construction is the condition
\eqref{maslov4}, hence we may in fact find a single-valued $U$
microlocally unitary along a neighborhood of $T$ such that
\eqref{betteregorov} holds for $a$ with essential support in a
neighborhood of $T.$

We may of course employ the same construction over a neighborhood of any translate of
$T$ inside $\lag.$ However, the resulting microlocally defined
operators may not fit together to be globally defined $U$ over $\lag:$
in general, the construction of \cite{Hitrik-Sjostrand1} now yields
$U$ that is Floquet-periodic with respect to cycles in $T'$ under the
splitting $\lag\cong T \times T'$ as described in the model setting.
This is the conjugating operator that we shall employ.

Now we finally turn to the proof of the theorem.  With $U$ as
constructed above, we obtain
$$
(U P U^*) U u=O(h^{2+\delta})
$$
microlocally near $\kappa(T);$ moreover, $UPU^*$ satisfies the
hypotheses of Proposition~\ref{proposition:normformcase} in this
neighborhood, with the condition on the subprincipal symbol being
guaranteed by the improved Egorov property \eqref{betteregorov}, and
with the slight variation of taking values in a flat bundle over
$\lag_0$ with trivial holonomy along cycles lying in
$T.$\footnote{This makes virtually no difference in the proof of
  Proposition~\ref{proposition:normformcase}: the only change is that
  the elliptic operator to which we apply the unique continuation
  theorem acts on sections of a flat ($h$-dependent) complex line
  bundle.  This does not affect our application of
  Theorem~17.2.6 of \cite{Hormander:vol3}, which is a merely local
  statement.  The constant $c$ in our quantatitive statement of unique
  continuation \eqref{uniquecontinuation} can be taken uniform with
  respect to the compact set of possible cycles defining such a
  bundle, hence the same argument applies as in the scalar case.}  Thus, the principal
symbol of $U(u)$ must satisfy the unique continuation property.\qed

\bibliography{all}
\bibliographystyle{amsplain}
\end{document}